\documentclass[11pt,leqno]{amsart}
\usepackage[letterpaper, total={6in, 8in}, left=1.25in, top=1.5in]{geometry}
\usepackage{amsmath}
\usepackage{amsfonts}
\usepackage{amssymb}
\usepackage{graphicx}
\usepackage{color}
\usepackage{xcolor}
\usepackage{hyperref}
\newtheorem{theorem}{Theorem}[section]
\newtheorem*{theorem*}{Theorem}
\newtheorem{lemma}{Lemma}[section]
\newtheorem{corollary}[theorem]{Corollary}
\newtheorem{proposition}{Proposition}[section]

\newtheorem{definition}[theorem]{Definition}

\def\Ric{\text{Ric}}

\def\p{\partial}

\def\R{\mathbb{R}}

\def\vp{\varphi}

\def\k{\kappa}

\def\H{\mathbb{H}}

\def\Ric{\operatorname{Ric}}

\def\Sect{\operatorname{Sect}}

\def\diam{\operatorname{diam}}
\def\dist{\operatorname{dist}}

\newcommand{\eps}{{\varepsilon}}

\numberwithin{equation}{section}

\begin{document}

\title[An upper bound for the first nonzero Steklov eigenvalue]{An upper bound for the first nonzero Steklov eigenvalue}

\author{Xiaolong Li}
\address{Department of Mathematics, University of California, Irvine, CA 92697, USA}
\email{xiaolol1@uci.edu}

\author{Kui Wang}
\address{School of Mathematical Sciences, Soochow University, Suzhou, 215006, China}
\email{kuiwang@suda.edu.cn}

\author{Haotian Wu}
\address{School of Mathematics and Statistics, The University of Sydney, NSW 2006, Australia}
\email{haotian.wu@sydney.edu.au}

\date{\today}
\subjclass[2010]{35P15, 49R05, 58C40, 58J50}
\keywords{Steklov Eigenvalue, Brock-Weinstock inequality, Spherical Symmetrization}

\begin{abstract}  
Let $(M^n,g)$ be a complete simply connected $n$-dimensional Riemannian manifold with curvature bounds $\Sect_g\leq \kappa$ for $\kappa\leq 0$ and $\Ric_g\geq(n-1)Kg$ for $K\leq 0$. We prove that for any bounded domain $\Omega \subset M^n$ with diameter $d$ and Lipschitz boundary, if $\Omega^*$ is a geodesic ball in the simply connected space form with constant sectional curvature $\kappa$ enclosing the same volume as $\Omega$, then $\sigma_1(\Omega) \leq C \sigma_1(\Omega^*)$, where $\sigma_1(\Omega)$ and $ \sigma_1(\Omega^*)$ denote the first nonzero Steklov eigenvalues of $\Omega$ and $\Omega^*$ respectively, and $C=C(n,\kappa, K, d)$ is an explicit constant. When $\kappa=K$, we have $C=1$ and recover the Brock-Weinstock inequality, asserting that geodesic balls uniquely maximize the first nonzero Steklov eigenvalue among domains of the same volume, in Euclidean space and the hyperbolic space. 
\end{abstract}

\maketitle


\section{Introduction}

Let $(M^n,g)$ be a complete Riemannian manifold of dimension $n$ and $\Omega \subset M^n$ be a bounded domain with Lipschitz boundary. The Steklov eigenvalue problem is to find a solution $u$ of the boundary value problem 
\begin{align*}
\begin{cases} 
\Delta u =0 & \text{ in } \Omega,\\
\frac{\p u}{\p \nu} =\sigma u & \text{ on } \p \Omega,
\end{cases}
\end{align*}
where $\Delta$ denotes the Laplace-Beltrami operator, $\nu$ denotes the outward unit normal to $\p \Omega$, and $\sigma$ is a real number. This problem was first introduced by Steklov \cite{Steklov02} in 1902 for bounded domains in the plane.
The set of eigenvalues for the Steklov problem is the same as that for the well-known Dirichlet-to-Neumann map, which maps $f\in L^2(\p \Omega)$ to the normal derivative on the boundary of the harmonic extension of $f$ inside $\Omega$. 
Since the Dirichlet-to-Neumann map is a self-adjoint operator, it has a discrete spectrum given by
\begin{align*}
0=\sigma_0(\Omega)<\sigma_1(\Omega ) \le \sigma_2(\Omega) \le \cdots \rightarrow \infty.
\end{align*}
The eigenfunctions of $\sigma_0(\Omega)$ are the constant functions. The first nonzero eigenvalue $\sigma_1(\Omega)$ is characterized by the following Rayleigh quotient
\begin{align}\label{min-max}
\sigma_1(\Omega)=\inf\left\{ \frac{\int_\Omega |\nabla u|^2 \ d\mu_g}{\int_{\partial \Omega}  u^2 \ dA_g} : u\in W^{1,2}(\Omega)\setminus\{0\},\;\; \int_{\partial \Omega} u\ dA_g=0 \right\},
\end{align}
where $d\mu_g$ is the volume form of $g$ and $dA_g$ is the induced measure on $\p \Omega$.

In 1954, Weinstock \cite{Weinstock54} showed that the round disk uniquely maximizes $\sigma_1(\Omega)$ among simply connected planar domains with prescribed perimeter. This result was generalized to arbitrary compact Riemannian surfaces by Fraser and Schoen \cite{FS11} to obtain the upper bound $\sigma_1(\Omega) |\p \Omega| \leq 2 \pi (\gamma +k)$ for a surface of genus $\gamma$ with $k$ boundary components. In higher dimensions, Bucur, Ferone, Nitsch and Trombetti \cite{BFNT17} proved that the ball uniquely maximizes $\sigma_1(\Omega)$ among bounded open convex sets in $\R^n$ with prescribed perimeter. The convexity assumption in the previous result is crucial. Indeed, for an annulus $B_1(0) \setminus B_\eps(0)$ with $\eps$ sufficiently small, its first nonzero Steklov eigenvalue is strictly bigger than that of a ball with same volume, see \cite{GP17}. Also, Fraser and Schoen \cite{FS19} have shown that the ball does not maximize $\sigma_1(\Omega)$ among contractible domains in $\R^n$ with prescribed perimeter. Moreover, they have given an explicit upper bound on $\sigma_1(\Omega)$ for any smooth domain in $\R^n$ in terms of its boundary perimeter (cf. \cite[Section 2]{FS19}).

When combined with the isoperimetric inequality, Weinstock's theorem implies that the round disk uniquely maximizes $\sigma_1(\Omega)$ among all simply connected planar domains with fixed area. In 2001, Brock \cite{Brock01} generalized Weinstock's result by removing any topological or dimensional restriction. As a result, we have the Brock-Weinstock inequality, which asserts that among domains in $\mathbb{R}^n$ with the same volume, the ball maximizes $\sigma_1(\Omega)$, and the equality occurs if and only if $\Omega$ is a ball. A sharp quantitative version of the Brock-Weinstock inequality has been proved by Brasco, De Philippis and Ruffini \cite{BDPR12}.

The Brock-Weinstock inequality is related to two classic spectral inequalities: the Faber-Krahn inequality, which asserts that the ball uniquely minimizes the first Dirichlet eigenvalue among domains with the same volume, and the Szeg\"o-Weinberger inequality stating that among domains with the same volume, the ball uniquely maximizes the first nonzero Neumann eigenvalue. It is well-known that the Faber–Krahn inequality holds in any Riemannian manifold in which the isoperimetric inequality holds, see \cite{Chavel84}. Also, the Sz\"ego-Weinberger inequality holds for domains in the hemisphere and in the hyperbolic space \cite{AB95}. Therefore, it is a natural question to extend the Brock-Weinstock inequality to space forms and more general Riemannian manifolds. 

Concerning the previous question, only a few results are known. In 1999, Escobar \cite{Escobar99} generalized Weinstock's theorem by proving that in a complete simply connected two-dimensional manifold with constant Gaussian curvature, geodesic balls maximize $\sigma_1(\Omega)$ among bounded simply connected domains with fixed area. In the same paper, the author obtained the more general eigenvalue comparison result: $\sigma_1(\Omega)$ of any bounded simply connected domain in a complete simply connected non-positively curved two-manifold is no larger than that of a ball in $\R^2$ with the same area, and the equality holds only when the domain is isometric to the round disk. In 2014, Binoy and Santhanam \cite{BS14} proved that in non-compact rank one symmetric spaces (including Euclidean space and hyperbolic space), geodesic balls maximize $\sigma_1(\Omega)$ among bounded domains of the same volume. Recently, a stability result for the theorem of Binoy and Santhanam has been proved by Castillon and Ruffini \cite{CR16}. 

The main purpose of this paper is to give an upper bound for
the first nonzero Steklov eigenvalue of a bounded domain in a 
simply connected Riemannian manifold $(M^n,g)$ with non-positive sectional curvatures. Throughout the paper, the function $sn_\kappa$ is defined by
\begin{align}\label{sine}
sn_\kappa(t) := 
\begin{cases}
\frac{1}{\sqrt{\kappa}} \sin(\sqrt{\kappa} t), & \text{ if } \kappa >0,\\
t, & \text{ if } \kappa=0,\\
\frac{1}{\sqrt{-\kappa}}\sinh{(\sqrt{-\kappa}t)}, & \text{ if } \kappa<0.
\end{cases}
\end{align}
We denote by $\Sect_g$ and $\Ric_g$ the sectional curvature and the Ricci curvature of $g$ respectively, and by $\diam(\Omega)$ the diameter of $\Omega\subset M^n$.

The main theorem of this paper states the following.
\begin{theorem}\label{th1}
	Let $(M^n, g)$ be a complete simply connected Riemannian manifold of dimension $n$, and $\Omega\subset M^n$ be a bounded domain with Lipschitz boundary. Let $M_\kappa$ be the $n$-dimensional simply connected space form of constant sectional curvature $\kappa$, and $\Omega^*$ be a geodesic ball in $M_{\kappa}$ having the same volume as $\Omega$. If $\Sect_g\leq\kappa$ for $\kappa\le 0$, and $\Ric_g\geq (n-1)Kg$ for $K\leq 0$, then
	\begin{align}\label{main-inequality}
	\sigma_1(\Omega) \le \left(\frac{sn_K (d)}{sn_\kappa (d)}\right)^{2n-2}\sigma_1(\Omega^*),
	\end{align}
	where $d=\operatorname{diam}(\Omega)$.
\end{theorem}

In Euclidean space or hyperbolic space, we have $k=K$ and the constant factor in \eqref{main-inequality} is $1$. So Theorem \ref{th1} recovers the Brock-Weinstock inequality proved by Weinstock \cite{Weinstock54} and \cite{Brock01} for $\mathbb{R}^n$, and by Binoy and Santhanam \cite{BS14} for $\H^n$.
\begin{corollary}\label{cor 1.2}
	In Euclidean space and hyperbolic space, geodesic balls uniquely maximize the first nonzero Steklov eigenvalue among bounded Lipschitz domains with the same volume. 
\end{corollary}

We note that Corollary \ref{cor 1.2} has been generalized to the Robin eigenvalues of the Laplacian in non-positively curved space forms by the authors \cite{LWW20-robin}.

On manifolds whose sectional curvatures are bounded from above by $\k$, where $\k\le0$, Binoy and Santhanam obtained a result (cf. \cite[Theorem 1.2]{BS14}) similar to Theorem \ref{th1}. The constant in their inequality depends on the manifold and the space form in comparison, although in a rather non-transparent way. In contrast, the constant in our inequality \ref{main-inequality} reveals the explicit dependency on the geometries.

When $\kappa=0$, it is well-known that
\begin{align*}
\sigma_1(\Omega^*)=\left(\frac{\omega_n}{\operatorname{Vol}(\Omega)}\right)^{1/n},
\end{align*}
where $\omega_n$ is volume of the unit ball in $\mathbb{R}^n$. Then Theorem \ref{th1} gives the following explicit estimate in a Cartan–Hadamard manifold, i.e., a complete simply connected Riemannian manifold with non-positive sectional curvature.
\begin{corollary}
	Let $(M^n, g)$ be a Cartan–Hadamard manifold of dimension $n$, and $\Omega\subset M^n$ be a bounded domain with Lipschitz boundary. If $\Ric_g\geq(n-1)Kg$ for $K\leq 0$, then
	\begin{align}\label{cor 1.3}
	\sigma_1(\Omega)\operatorname{Vol}(\Omega)^{1/n}\leq \omega_n^{1/n}\left(\frac{sn_K(d)}{d}\right)^{2n-2},
	\end{align}
	where $d=\operatorname{diam}(\Omega)$.
\end{corollary}

To conclude this section, we mention several other aspects of the first nonzero Steklov eigenvalue $\sigma_1(\Omega)$. First of all, the question of finding a metric on $\Omega$ maximizing $\sigma_1(\Omega)|\p \Omega|$ has received considerable attention in recent years since the remarkable paper by Fraser and Schoen \cite{FS16}, in which the authors developed the theory of extremal metrics for Steklov eigenvalues via its connection to the free boundary minimal surfaces. Secondly, finding a lower bound for $\sigma_1(\Omega)$ in terms of the geometric data of $\Omega$ is also an interesting question. In this direction, Escobar \cite{Escobar97} proved that for an $n$-dimensional ($n\ge 3$) compact smooth Riemannian manifold with boundary, which has non-negative Ricci curvature and the principal curvatures of the boundary bounded below by $c>0$, the first nonzero Steklov eigenvalue is greater than or equal to $c/2$. Escobar then conjectured in \cite{Escobar99} that the sharp lower bound is $c$ with the equality being true only on isometrically Euclidean balls with radius $1/c$. Recently, Xia and Xiong \cite{XX19} settled Escobar's conjecture under the stronger assumption of non-negative sectional curvature. Lastly, $\sigma_1(\Omega)$ is closely related to the first nonzero Laplace eigenvalue of $\p \Omega$. We refer the reader to the papers by Wang and Xia \cite{WX09}, Karpukhin \cite{Karpukhin17}, Xiong \cite{Xiong18}, Xia and Xiong \cite{XX19} for recent developments. 

This paper is organized as follows. In Section \ref{prelim}, we set up the notation and recall some facts on the eigenfunctions for the first nonzero Steklov eigenvalue on space forms. Section \ref{lemmas} contains results on spherical symmetrizations and the comparison of isoperimetric profiles. We prove Theorem \ref{th1} in Section \ref{proofmain}.

\subsection*{Acknowledgements}
We thank Professors Richard Schoen, Lei Ni and Zhou Zhang for their encouragement and support. 
K. Wang is partially supported by NSFC No.11601359; H. Wu is supported by ARC Grant DE180101348. Both K. Wang and H. Wu acknowledge the excellent work environment provided by the Sydney Mathematical Research Institute.

\section{Preliminaries}\label{prelim}
For any bounded Lipschitz domain $\Omega\subset M:=M^n$, we denote by $|\Omega|$ and $|\partial \Omega|$ the $n$-dimensional volume of $\Omega$ and
the $(n-1)$-dimensional Hausdorff measure of $\partial \Omega$ respectively, each taken with respect to the Riemannian metric $g$ on $M$. Let $(M_\kappa,g_\kappa)$ denote the $n$-dimensional complete simply connected space form of constant sectional curvature $\kappa$, and $\Omega^*_q$ be a geodesic ball in $M_\kappa$ centered at $q$ and satisfying $|\Omega^\ast_q|_\kappa=|\Omega|$, where $|\Omega^*_q|_\kappa$ is the $n$-dimensional volume of $\Omega^*$ with respect to $g_\kappa$.

\subsection{Steklov eigenfunctions on space forms} 

In this subsection, we collect some known facts on the Steklov eigenfunctions corresponding to $\sigma_1(\Omega_q^*)$. Let $R_0$ be the radius of the geodesic ball $\Omega_q^*$ in $M_\kappa$, and $(r, \theta)$ be the polar coordinates centered at $q$. 
Recall that the eigenfunctions on $\Omega_q^*$ corresponding to $\sigma_1(\Omega_q^*)$ are given by
$$u_i(r, \theta)=F(r)\psi_i(\theta),\text{\quad\quad\quad}1\le i\le n,$$
where $\psi_i(\theta)$ are linear coordinate functions restricted to $\mathbb{S}^{n-1}$, and $F(r)$ solves the following ODE initial value problem
\begin{align}\label{EqF}
F''(r)+(n-1)\frac{sn'_\kappa(r)}{sn_\kappa(r)} F'(r) -\frac{n-1}{sn^2_\kappa(r)}F(r)=0,\quad F(0)=0,\quad F'(0)=1.
\end{align}
Then $F>0$ on $(0,\infty)$ by the maximum principle. Also, $F'(r)>0$ on $(0,\infty)$. Indeed, using \eqref{EqF}, we calculate
\begin{align*}
\left(sn_{\kappa}^{n-1}F'\right)' &= sn_\kappa^{n-1} F'' +(n-1)sn_\kappa^{n-2} sn_\kappa' F' \\
&= sn_\kappa^{n-1}(n-1) \left(-\frac{sn'_\kappa F'}{sn_\kappa} +\frac{F}{sn_\kappa^2} + \frac{sn'_\kappa F'}{sn_\kappa}\right)\\
&= (n-1)sn_\kappa^{n-3} F\\
&> 0.
\end{align*}
So then $sn_\kappa^{n-1}(r)F'(r) > sn_\kappa^{n-1}(0)F'(0)=0$, implying that $F'(r)>0$ on $(0,\infty)$. It is straightforward to check that $\sigma_1(\Omega_q^*)=F'(R_0)/F(R_0)$ is the minimal value of the quotient
\begin{align*}
Q(\vp)=\frac{\int_0^{R_0}\left((\vp'(r))^2+\frac{n-1}{sn^{n-1}_\kappa(r)}\vp^2(r)\right)sn_\kappa^{n-1}(r)\,dr}{\vp^2(R_0)sn_\kappa^{n-1}(R_0)}\quad \text{with } \vp(0)=0.
\end{align*}

By calculating the first derivatives and using the differential equation (\ref{EqF}), we have the following monotonicity results.
\begin{proposition}\label{ProM}
	Let $F(r)$ be the function defined in equation \eqref{EqF}. Define 
	\begin{align}
	G(r)&:= (F^2(r))'+\frac{(n-1)sn'_\kappa(r)}{sn_\kappa(r)}F^2(r), \label{def-g}\\
	H(r)&:= (F'(r))^2+\frac{n-1}{sn_\kappa^2(r)}F^2(r). \label{def-h}
	\end{align}
	Then $G$ is non-negative and non-decreasing on $[0,\infty)$ for all $\kappa \in \R$, and $H$ is non-negative and non-increasing on $[0,\infty)$ provided that $\kappa \leq 0$. 
\end{proposition}
\begin{proof}
	The functions $G$ and $H$ are non-negative on $[0,\infty)$ since $F$ is non-negative and increasing on $[0, \infty)$.
	
	Using equation \eqref{EqF}, we calculate on $(0,\infty)$ that
	\begin{align*}
	G'(r)&= 2FF'' +2(F')^2+(n-1)\left(2\frac{sn_\kappa'}{sn_\kappa}FF'+\frac{sn_\kappa''}{sn_\kappa} F^2 -\frac{(sn_\kappa')^2}{sn_\kappa^2}F^2  \right)\\
	&= 2F \left(-(n-1)\frac{sn'_\kappa}{sn_\kappa}F' +\frac{n-1}{sn_\kappa^2} F \right) +2(F')^2 \\ 
	&\quad +(n-1)\left(2\frac{sn_\kappa'}{sn_\kappa}FF'+\frac{sn_\kappa''}{sn_\kappa} F^2 -\frac{(sn_\kappa')^2}{sn_\kappa^2}F^2  \right)\\
	&= 2(F')^2+\frac{(n-1)F^2}{s_\kappa^2}\left(2+sn_\kappa sn_\kappa'' -(sn_\kappa')^2 \right)\\
	&= 2(F')^2+\frac{(n-1)F^2}{sn_\kappa^2}\\
	&\ge 0,
	\end{align*}
	where in the last equality we used the identity $sn_\kappa sn_\kappa''-(sn_\kappa')^2 =-1 $ for all $\kappa\in\mathbb{R}$. Thus, $G$ is non-decreasing on $(0,\infty)$. 
	
	Likewise, we have on $(0,\infty)$ that
	\begin{align*}
	H'&=2F'F''-2(n-1)\frac{sn_\kappa'}{sn_\kappa^3}F^2+\frac{2(n-1)}{sn_\kappa^2}FF'\\
	&=2F'\left(-(n-1)\frac{sn_\kappa'}{sn_\kappa}F'+\frac{n-1}{sn_\kappa^2}F\right)-2(n-1)\frac{sn_\kappa'}{sn_\kappa^3}F^2+\frac{2(n-1)}{sn_\kappa^2}FF'\\
	&= -\frac{2(n-1)}{sn_\kappa^3}\left(sn_\kappa'sn_\kappa^2(F')^2-2sn_\kappa FF'+sn_\kappa' F^2\right)\\
	&\le -\frac{2(n-1)}{sn_\kappa^3}\left(sn_\kappa^2(F')^2-2sn_\kappa FF'+ F^2 \right)\\
	&= -\frac{2(n-1)}{sn_\kappa^3}\left(sn_\kappa F' - F \right)^2\\
	&\le 0,
	\end{align*}
	where in the first inequality we used $sn_{\kappa}'(r) \geq 1$ for $\kappa \leq 0$. Thus, $H$ is non-increasing on $(0,\infty)$.
\end{proof}

\section{Spherical symmetrizations and isoperimetric inequality}\label{lemmas}

We recall the definitions of spherical symmetrizations. For any non-negative real-valued function $f$ defined on a bounded domain $\Omega\subset M$, the measure of the super-level sets of $f$ is defined by
\begin{align*}
\mu_f(t):=|\{x\in \Omega: f(x)> t\}|.
\end{align*} 
Let $r_q(x)=\operatorname{dist}_{\kappa}(q,x)$ be the distance function on the space form $M_\kappa$ and $B_q(r)$ be the geodesic ball centered at $q$ with radius $r$ in $M_\kappa$. 

\begin{definition} \label{fu}
	Let $\Omega\subset M$ be a bounded domain and $f$ be a non-negative integrable real-valued function defined on $\Omega$. The spherical decreasing and increasing symmetrizations of $f$, denoted by $f^*(x)$  and $f_*(x)$ respectively, are radial functions defined on $\Omega_q^*$ by
	\begin{align*}
	f^*(x):=\sup\left\{t:\mu_f(t)\ge|B_q(r_q(x))|_{\kappa}\right\}
	\end{align*}
	and
	\begin{align*}
	f_*(x):=\sup\left\{t:\mu_f(t)\ge|\Omega_q^*|_{\kappa}-|B_q(r_q(x))|_{\kappa}\right\},
	\end{align*}
	where $\Omega_q^*$ is the geodesic ball in $M_\kappa$ centered at $q$ satisfying $|\Omega_q^*|_{\kappa} = |\Omega|$.
\end{definition}

The $L^s$-norm ($s\geq 1$) is invariant under spherical symmetrizations.
\begin{proposition}\label{Fu}
	For any $s \ge1$, we have
	\begin{align}
	||f(x)||_{L^s(\Omega)}=||f^*(x)||_{L^s(\Omega^*_q)}=||f_*(x)||_{L^s(\Omega^*_q)}.\label{Ls}
	\end{align}
\end{proposition}
\begin{proof}
	See \cite[Proposition 2.2]{Edelen17}.
\end{proof}

For any $p\in M $, let $\eta_p: [0,\infty)\to[0,\infty)$ be the radial function defined by 
\begin{align}\label{def-eta}
\left\vert B_q(\eta_p(r))\right\vert_{\kappa}=|B_p(r)|. 
\end{align}
Clearly, $\eta_p$ is monotone non-decreasing in $r$. The volume comparison theorem for $\Sect_g\le \kappa$ implies that $\eta_p(r)\ge r$.

We first prove a center of mass result.
\begin{lemma}\label{lmtest}
	Assume that $(M^n,g)$ is complete simply connected with $\Sect_g\leq\kappa$ for $\kappa\leq 0$, and $\Omega\subset M^n$ is any bounded domain. Then there exists a point $p\in \operatorname{hull}(\Omega)$, the closed geodesic convex hull of $\Omega$, such that
	\begin{align*}
	\int_{\partial \Omega} (F\circ \eta_p\circ r_p)(x) \frac{\exp_p^{-1}(x)}{r_p(x)}\, dA_g=0,
	\end{align*}
	where $F$ is defined in equation \eqref{EqF}, $r_p(x)=\dist_g(p,x)$, and $\exp_p^{-1}(x)$ denotes the inverse of the exponential map $\exp_p:T_pM^n \to M^n$.
\end{lemma}

\begin{proof}
	The proof is similar to \cite[Lemma 4.1]{Edelen17}. 
	Define the vector field
	$$
	X(p)=\int_{\partial \Omega} (F\circ \eta_p\circ r_p)(x)\frac{\exp_p^{-1}(x)}{r_p(x)}\ dA_g.
	$$
	Then the integral curves of $X$ defines a mapping from $\operatorname{hull}(\Omega)$ to itself. 
	Since $\operatorname{hull}(\Omega)$ is convex and contained in the injectivity radius, $\operatorname{hull}(\Omega)$ is a topological ball and thus $X$ must have a zero by the Brouwer fixed point theorem.
\end{proof}

The spherical symmetrizations of monotone radial functions have the following properties.
\begin{lemma}\label{lmsp}
	Assume $f(r)$ is a non-negative function on $[0,\infty)$.
	\begin{enumerate}
		\item If $f(r)$ is non-decreasing, then for  $y\in\Omega_q^*$
		\begin{align}\label{Equ1}
		\left(f \circ \eta_p \circ r_p\right)_{*}(y)\ge f(r_q(y)).
		\end{align}
		\item If $f(r)$ is non-increasing, then for  $y\in\Omega_q^*$
		\begin{align}\label{Equ2}
		\left(f\circ \eta_p\circ r_p \right)^{*}(y)\le  f( r_q(y)).
		\end{align}
	\end{enumerate}
\end{lemma}
\begin{proof}
	It follows from the definitions of $\eta_p$ and spherical symmetrizations that
	\begin{align}\label{pf1}
	\left(f\circ \eta_p\circ r_p\right)_{*}(y)=f(\eta_p(r_1)),
	\end{align}
	where $r_1$ satisfies
	$|B_q(r_q(y))|_{\kappa}=|B_p(r_1)\bigcap \Omega|\leq |B_q(\eta_p(r_1))|_{\kappa}$. So then
	\begin{align*}
	r_q(x)\le \eta_p(r_1),
	\end{align*}
	which implies (\ref{Equ1}) since $f$ is non-decreasing.
	
	The proof of (\ref{Equ2}) is similar as that of (\ref{Equ1}) and we omit the details.
\end{proof}

We now prove a comparison result for isoperimetric profiles.

\begin{lemma}\label{isop}
	Assume that $(M^n,g)$ is complete simply connected with $\Sect_g\leq\kappa$ for $\kappa\leq 0$. Fix $p \in M^n$ and define an isoperimetric profile $I_M:[0, \infty) \to \R^+$ by
	\begin{align*}
	I_M(t):=\operatorname{Area}\left(\p B_{r(t)}(p)\right),
	\end{align*}
	where $r(t)$ is so defined that $\operatorname{Vol}(B_{r(t)}(p))=t$. Then 
	\begin{equation}\label{iso}
	I_M(t)\ge I_{M_\kappa}(t).
	\end{equation}
\end{lemma}
\begin{proof}
	Let $r_1$ and $r_2$ satisfy
	\begin{align*}
	t=\int_{\mathbb{S}^{n-1}}\int_0^{r_1}J(r,\theta) \, dr \,d\theta=n\omega_n\int_0^{r_2}J_\kappa(r) \, dr.
	\end{align*}
	Since $\operatorname{Sect}_g\le \kappa$, we have the following comparisons
	\begin{align}\label{com}
	J(r,\theta)\ge J_\kappa(r) \text{\quad and \quad} \frac{J'(r,\theta)}{J(r,\theta)}\ge\frac{J_\kappa'(r)}{J_\kappa(r)}.
	\end{align}
	Then from the definitions of $r_1$ and $r_2$, we have  $r_1\le r_2$. By direct calculation,
	\begin{align*}
	I_M'(t)&=\frac{\int_{\mathbb{S}^{n-1}}J'(r_1,\theta)\, d\theta}{\int_{\mathbb{S}^{n-1}}J(r_1,\theta)\, d\theta}\ge\frac{\int_{\mathbb{S}^{n-1}}\frac{J_\kappa'(r_1)}{J_\kappa(r_1)}J(r_1,\theta)\, d\theta}{\int_{\mathbb{S}^{n-1}}J(r_1,\theta)\, d\theta}=\frac{J_\kappa'(r_1)}{J_\kappa(r_1)}
	\ge\frac{J_\kappa'(r_2)}{J_\kappa(r_2)},
	\end{align*}
	where we used the comparison \eqref{com} in the first inequality and $r_1\le r_2$ in the last inequality. Similar calculation shows that
	$$
	I_{M_\kappa}'(t)=\frac{J_\kappa'(r_2)}{J_\kappa(r_2)}.
	$$
	Therefore, we have $$I_M'(t)-I_{M_\kappa}'(t)\ge 0,$$
	thus implying the lemma.
\end{proof}

The next lemma estimates the derivative of $\eta_p(r)$ in terms of the curvatures and the diameter of $\Omega$.
\begin{lemma}\label{lm2}
	Assume that $(M^n,g)$ is complete simply connected with $\Sect_g\le \kappa$ for $\kappa\leq 0$ and $\Ric_g \ge (n-1)Kg$ for $K\leq 0$. Then for all $r\in(0,d]$, where $d=\diam(\Omega)$, we have
	\begin{align}
	\eta_p'(r) & \ge 1, \label{estimates2}\\
	\max\left\{\eta_p'(r), \frac{sn_\kappa\big(\eta_p(r)\big)}{sn_\kappa(r)}\right\} & \le \left(\frac{sn_K (d)}{sn_\kappa (d)}\right)^{n-1}. \label{estimates}
	\end{align}
\end{lemma}
\begin{proof}
	We write $\eta_p$ as $\eta$ for short.
	
	Since $\eta'(r)=\frac{d |B_r|}{dr}\frac{d \eta(r)}{d|B_r|}$,
	we have
	\begin{align*}
	\eta'(r)=\frac{|\partial B_r|}{m_\kappa'\left(m_\kappa^{-1}(|B_r|)\right)},
	\end{align*}
	where $m_\kappa(r)=|B_r|_\kappa$. By the definition \eqref{def-eta} of $\eta(r)$, we see that $\eta(r)=m_\kappa^{-1}(|B_r|)$.
	So then
	\begin{align*}
	m_\kappa'\left(m_\kappa^{-1}(|B_r|)\right)=m_\kappa'(\eta(r))=|\partial B_{\eta(r)}|_{\kappa},
	\end{align*}
	which gives
	\begin{align*}
	\eta'(r)=\frac{|\partial B_{r}|}{|\partial B_{\eta(r)}|_{\kappa}}.
	\end{align*}
	Since $\Sect_g\le \kappa$, then from the isoperimetric inequality \eqref{iso}, we deduce
	\begin{align*}
	|\partial B_{\eta(r)}|_{\kappa}\le |\partial B_{r}|,
	\end{align*}
	thus proving (\ref{estimates2}).
	
	Inequality (\ref{estimates}) has been proven in \cite[page 863]{Edelen17}. We give a different proof here.  Since $\eta(r)\ge r$, we have
	\begin{align}\label{eqeta1}
	\eta'(r)=\frac{|\partial B_{r}|}{|\partial B_{\eta(r)}|_{\kappa}}\le\frac{|\partial B_{r}|}{|\partial B_{r}|_{\kappa}} \le \frac{|\partial B_{r}|_{K}}{|\partial B_{r}|_{\kappa}} \le \left(\frac{sn_K (d)}{sn_\kappa (d)}\right)^{n-1},
	\end{align}
	where we have used the curvature condition  $\operatorname{Ric}_g\ge (n-1)Kg$ and the fact that $\frac{sn_K(r)}{sn_\kappa(r)}$ is non-decreasing in $r$.  
	
	Using the isoperimetric inequality (\ref{iso}), we estimate that
	\begin{align*}
	\frac{sn_\kappa\big(\eta(r)\big)}{sn_\kappa(r)}=\left(\frac{|\partial B_{\eta(r)}|_{\kappa}}{|\partial B_r|_{\kappa}}  \right)^{\frac {1} {n-1}}\le \left(\frac{|\partial B_{r}|}{|\partial B_r|_{\kappa}}  \right)^{\frac {1} {n-1}}.
	\end{align*}
	Since $\Ric_g\ge(n-1)Kg$, we have $|\partial B_1|\le|\partial B_r|_{\kappa}$. Therefore, we get
	\begin{align}\label{eqeta2}
	\frac{sn_\kappa\big(\eta(r)\big)}{sn_\kappa(r)}\le \left(\frac{|\partial B_{r}|_{K}}{|\partial B_r|_{\kappa}}  \right)^{\frac {1} {n-1}}=\frac{sn_K(r)}{sn_\kappa(r)}\le \frac{sn_K(d)}{sn_\kappa(d)}
	\end{align}
	where we have again used that $\frac{sn_K(r)}{sn_\kappa(r)}$ is non-decreasing in $r$. Then (\ref{estimates}) follows from (\ref{eqeta1}) and (\ref{eqeta2}).
	
	Therefore, the lemma is proved.
\end{proof}

\section{Proof of Theorem \ref{th1}}\label{proofmain}
We divide the proof of Theorem \ref{th1} into four propositions, each of which gives a different upper bound for $\sigma_1(\Omega)$ and might be of independent interest. 

From here on, we fix $p\in \operatorname{hull}(\Omega)$ according to Lemma \ref{lmtest} so that
\begin{align}\label{test}
\int_{\partial \Omega}  \frac{\exp_p^{-1}(x)}{r_p(x)}\big(F\circ \eta_p\circ r_p\big)(x)\ dA_g= 0.
\end{align}
We denote by $(r, \theta)$, where $\theta\in\mathbb{S}^{n-1}$, the polar coordinates centered at $p$ and by $J(r, \theta)drd\theta$  the volume element at $(r,\theta)$. Then we have
$$
\frac{\exp_p^{-1}(x)}{r_p(x)}=(\psi_1(\theta),\psi_2(\theta),\cdots, \psi_n(\theta)),
$$
where  $\psi_i(\theta)$'s are the restrictions of the linear coordinate functions on $\mathbb{S}^{n-1}$.
We define
\begin{align*}
v_i := \left(F\circ\eta_p\circ r_p\right) \psi_i(\theta),\quad 1\le i\le n.
\end{align*}
Then \eqref{test} is equivalent to
$$
\int_{\partial \Omega} v_i \ dA_g=0,\quad 1\le i\le n.
$$

Using $v_i$'s as test functions for $\sigma_1(\Omega)$, we obtain the following proposition.

\begin{proposition}\label{prop 4.1}
	Assuming the hypotheses of Theorem \ref{th1}, then
	\begin{align}\label{sig11}
	\sigma_1(\Omega) \le \frac{\int_{\Omega} \left( \left|F'(\eta_p(r_p))\eta_p'(r_p)\right|^2+\frac{n-1}{sn_\kappa^2(r_p)}F^2(\eta_p(r_p)) \right)\, d\mu_g}{\int_{\p \Omega} | F(\eta_p(r_p))|^2\, dA_g}.
	\end{align}
\end{proposition}
\begin{proof} 
	We write $\eta_p$ and $r_p$ as $\eta$ and $r$ for short.	
	
	We denote by $\nabla^{\mathbb{S}^{n-1}}$ the covariant derivative  with respect to the standard metric on the unit sphere $\mathbb{S}^{n-1}$, and by $\nabla$ the covariant derivative with respect to the metric $g=dr^2+g_{ij}(r,\theta) d\theta^i d\theta^j$ on $M$. Using 
	\begin{align*}
	\sum_{i=1}^n\psi^2_i=1 \text{\quad and \quad} \sum_{i=1}^n |\nabla^{\mathbb{S}^{n-1}} \psi_i|^2=n-1,
	\end{align*}
	we compute that 
	\begin{align}\label{eq 4.3}
	\sum_{i=1}^n \int_{\Omega} \left\vert\nabla  v_i\right\vert^2 \,d\mu_g
	&= \sum_{i=1}^n \int_{\Omega} \left|\nabla  \left(F(\eta(r)) \psi_i\right)\right|^2\,d\mu_g \nonumber \\
	&=\sum_{i=1}^n\int_{\Omega} \left\{ \left|F'(\eta(r))\eta'(r)\right|^2  \psi^2_i+ \frac{F^2(\eta(r))}{J^{\frac{2}{n-1}}(r, \theta)} |\nabla^{\mathbb{S}^{n-1}} \psi_i|^2 \right\} \,d\mu_g\nonumber \\
	&=\int_{\Omega} \left\{\left\vert F'(\eta(r))\eta'(r)\right\vert^2 +F^2(\eta(r)) \frac{n-1}{J^{\frac{2}{n-1}}(r, \theta)}\right\} \, d\mu_g  \nonumber \\
	&\leq \int_{\Omega} \left\{\left|F'(\eta(r))\eta'(r)\right|^2+\frac{n-1}{sn_\kappa^2(r)}F^2(\eta(r)) \right\} \, d\mu_g,
	\end{align}
	where in the last step we used 
	\begin{align*}
	J(r,\theta)\ge sn^{n-1}_\kappa(r),
	\end{align*}
	which follows from the Rauch comparison theorem. We also have
	\begin{align}\label{eq 4.4}
	\sum_{i=1}^n \int_{\p \Omega} v_i^2 \ dA_g =\sum_{i=1}^n \int_{\p \Omega} |F(\eta(r))|^2 \psi_i^2\ dA_g  =\int_{\p \Omega} |F(\eta(r))|^2\ dA_g .
	\end{align}
	So using the averaging of Rayleigh quotients for $v_i$, \eqref{eq 4.3} and \eqref{eq 4.4}, we obtain 
	\begin{align*}
	\sigma_1(\Omega) &\le  \frac{\sum\limits_{i=1}^n \int_{\Omega} |\nabla v_i|^2 \, d\mu_g}{\sum\limits_{i=1}^n  \int_{\p \Omega} v_i^2\, dA_g}
	\le   \frac{\int_{\Omega} \left(\left|F'(\eta(r))\eta'(r)\right|^2+\frac{n-1}{sn_\kappa^2(r)}F^2(\eta(r)) \right)\, d\mu_g}{\int_{\p \Omega} | F(\eta(r))|^2\, dA_g}.
	\end{align*}
	This proves the proposition. 
\end{proof}

\begin{proposition}\label{prop 4.2}
	Assuming the hypotheses of Theorem \ref{th1}, then for functions $G$ and $H$ defined in Proposition \ref{ProM}, there holds
	\begin{align}
	\sigma_1(\Omega)\le \left(\frac{sn_K(d)}{sn_\kappa (d)}\right)^{2n-2}\ \frac{\int_{\Omega} H(\eta_p(r_p)) \, d\mu_g}{\int_\Omega  G(\eta_p(r_p)) \, d\mu_g},\label{7}
	\end{align}
	where $d=\diam(\Omega)$.
\end{proposition}

\begin{proof}
	We write $\eta_p$ and $r_p$ as $\eta$ and $r$ for short.	
	
	It follows from the definition \eqref{def-h} of $H$ and the estimate \eqref{estimates} in Lemma \ref{lm2} that
	\begin{align}\label{sig1}
	&\quad\, \left|F'(\eta(r))\right|^2(\eta'(r))^2+\frac{n-1}{sn_\kappa^2 (r)}F^2(\eta(r))\nonumber \\
	&\leq  \max\left\{(\eta'(r))^2, \frac{sn^2_\kappa\big(\eta(r)\big)}{sn^2_\kappa(r)}\right\} \left( \left|F'(\eta(r))\right|^2  +\frac{n-1}{sn_\kappa^2 (\eta(r))}F^2(\eta(r)) \right) \nonumber\\
	&\le  \left(\frac{sn_K(d)}{sn_\kappa (d)}\right)^{2n-2}
	H(\eta(r)).
	\end{align}
	
	We estimate the boundary integral.
	\begin{align}\label{sig2}
	& \quad\, \int_{\partial \Omega} F^2(\eta(r)) \, dA_g \nonumber\\
	&\ge \int_{\partial \Omega} F^2(\eta(r))\langle \nabla r, \nu \rangle \, dA_g\nonumber\\
	&=\int_\Omega \operatorname{div}\left(F^2(\eta(r))\nabla r\right) \, d\mu_g\nonumber\\
	&=\int_\Omega \left\{ (F^2)'(\eta(r))\eta'(r)+F^2(\eta(r))\Delta r \right\} d\mu_g\nonumber\\
	&\ge \int_\Omega \left\{ (F^2)'(\eta(r))\eta'(r)+\frac{(n-1)sn_\kappa'(r)}{sn_\kappa(r)}F^2(\eta(r))\right\}\, d\mu_g\nonumber\\
	&=\int_\Omega \left\{ (F^2)'(\eta(r))\eta'(r)+\frac{(n-1)sn_\kappa'(\eta(r))}{sn_\kappa(\eta(r))}F^2(\eta(r))\frac{sn_\kappa(\eta(r))}{sn'_\kappa(\eta(r))}\frac{sn_\kappa'(r)}{sn_\kappa(r)} \right\}\, d\mu_g\nonumber\\
	&\ge\int_\Omega \left\{ (F^2)'(\eta(r))+\frac{(n-1)sn_\kappa'(\eta(r))}{sn_\kappa(\eta(r))}F^2(\eta(r))\right\}\, d\mu_g\nonumber\\
	&=\int_\Omega G(\eta(r))\, d\mu_g,
	\end{align}
	where the last equality follows from the definition \eqref{def-g} of $G$, in the first inequality we used $|\nabla r| =1$, in the second inequality we used the Laplacian comparison theorem for the distance function, and in the last inequality we used $\eta'(r)\ge 1$ from Lemma \ref{lm2} and 
	\begin{equation*}
	\frac{sn_\kappa(\eta(r))}{sn'_\kappa(\eta(r))}\frac{sn_\kappa'(r)}{sn_\kappa(r)} \geq 1, 
	\end{equation*}
	which follows from $\eta(r)\ge r$ and that $\frac{sn_\kappa'(r)}{sn_\kappa(r)}$ is monotonically decreasing in $r$.  
	
	The proposition follows by substituting \eqref{sig1} and \eqref{sig2} into \eqref{sig11}.  
\end{proof}

Let $d\mu$ denote the volume form with respect to $g_\k$ on the space form $M_\kappa$.

\begin{proposition}\label{propsigma1}
	Assuming the hypotheses of Theorem \ref{th1}, then for functions $g$ and $h$ defined in Proposition \ref{ProM}, there holds
	\begin{align}\label{prop4.3}
	\sigma_1(\Omega)\le \left(\frac{sn_K(d)}{sn_\kappa (d)}\right)^{2n-2} \frac{\int_{\Omega_q^*} H(r_q)\, d \mu}{\int_{\Omega_q^*} G(r_q)\, d\mu}.
	\end{align}
\end{proposition}

\begin{proof}
	Lemma \ref{lmsp} applies to functions $g$ and $h$ defined in Proposition \ref{ProM}. Setting $f=G$ in inequality \eqref{Equ1} and $f=H$ in inequality \eqref{Equ2}, and using Proposition \ref{Fu}, we obtain
	\begin{align}\label{10}
	\int_\Omega H\circ \eta_p \circ  r_p \, d\mu_g=\int_{\Omega_q^*} \left(H\circ \eta_p \circ  r_p\right)^{*} \, d\mu\le \int_{\Omega_q^*} H(r_q)\, d\mu
	\end{align}
	and
	\begin{align}\label{11}
	\int_\Omega G\circ \eta_p \circ  r_p \, d\mu_g = \int_{\Omega_q^*} \left(G\circ \eta_p \circ  r_p\right)_{*}\, d\mu \ge \int_{\Omega_q^*} G(r_q)\, d\mu.
	\end{align}
	Assembling (\ref{7}), (\ref{10}) and (\ref{11}) together, we conclude the proposition. 
\end{proof}

\begin{proposition}\label{propsigma2}
	Assuming the hypotheses of Theorem \ref{th1}, then
	\begin{eqnarray}\label{prop4.4}
	\sigma_1(\Omega_q^*) =\frac{\int_{\Omega_q^*} H(r_q)\, d\mu}{\int_{\Omega_q^*} G(r_q)\, d\mu}.
	\end{eqnarray}
\end{proposition}
\begin{proof}
	Recall that $F(r)\psi_i(\theta)$, $1\le i\le n$, are the eigenfunctions for $\sigma_1(\Omega_q^*)$. It then follows that
	\begin{align*}
	\sigma_1(\Omega_q^*)=\frac{\int_{\Omega_q^*}\left( |F'|^2(r_q)+\frac{m-1}{sn_\kappa^2 (r_q)}F^2(r_q)\right)\, d\mu}{\int_{\partial \Omega_q^*} F^2(r_q)\, dA}=\frac{\int_{\Omega_q^*} H(r_q)\, d\mu}{\int_{\partial \Omega_q^*} F^2(r_q)\, dA},
	\end{align*}
	where $dA$  is the induced measure on $\p \Omega_q^*$. Also recalling the definition \eqref{def-g} of $G$ in Proposition \ref{ProM}, then we have
	\begin{align*}
	\int_{\partial \Omega_q^*} F^2(r_q)\, dA &= \int_{\partial \Omega_q^*} \langle F^2(r_q)\nabla r_q, \nu\rangle \, dA\\
	&=\int_{\Omega_q^*} \operatorname{div}\left( F^2(r_q)\nabla r_q \right) \, d\mu\\
	&=\int_{\Omega_q^*} \left( (F^2)'+F^2\Delta r_q \right) \, d\mu\\
	&=\int_{\Omega_q^*} \left( (F^2)'+\frac{(n-1)sn'_\kappa}{sn_\kappa}F^2 \right)\, d\mu\\
	&=\int_{\Omega_q^*} G(r_q)\, d\mu.
	\end{align*}
	Therefore, we have proved the proposition.
\end{proof}

\begin{proof}[Proof of Theorem \ref{th1}]
	Theorem \ref{th1} follows immediately from \eqref{prop4.3} and \eqref{prop4.4}. 
\end{proof}



\end{document}